\documentclass[10pt]{amsart}

\usepackage{amsfonts,amssymb,amscd,amstext}
\usepackage[charter]{mathdesign}
\usepackage[paper size={210mm,297mm},left=34mm,right=34mm,top=43.5mm,bottom=43.5mm]{geometry}
\usepackage[colorlinks=true,linkcolor=blue,citecolor=red]{hyperref}
\usepackage{graphicx}
\usepackage[utf8]{inputenc}
\usepackage{esint}
\usepackage{enumerate}

\renewcommand{\le}{\leqslant}
\renewcommand{\ge}{\geqslant}
\newcommand{\ptl}{\partial}

\newcommand{\rr}{{\mathbb{R}}}

\newcommand{\la}{\lambda}

\newcommand{\hh}{{\mathbb{H}}}

\newcommand{\hhh}{\mathcal{H}}

\newcommand{\escpr}[1]{g(#1)}
\newcommand{\Sg}{\Sigma}

\newcommand{\Om}{\Omega}
\newcommand{\eps}{\varepsilon}

\newcommand{\ga}{\gamma}
\newcommand{\Ga}{\Gamma}

\newcommand{\nuh}{\nu_{h}}

\newcommand{\n}{\nabla}

\DeclareMathOperator{\divv}{div}
\DeclareMathOperator{\tor}{Tor}

\DeclareMathOperator{\Jac}{Jac}

\newtheorem{theorem}{Theorem}[section]
\newtheorem{proposition}[theorem]{Proposition}
\newtheorem{lemma}[theorem]{Lemma}
\newtheorem{corollary}[theorem]{Corollary}

\theoremstyle{definition}

\theoremstyle{remark}

\newtheorem{remark}[theorem]{Remark}

\numberwithin{equation}{section}

\begin{document}

\title[Regularity of $C^1$ surfaces with prescribed mean curvature]{Regularity of $C^1$ surfaces with prescribed mean curvature\\ in three-dimensional contact sub-Riemannian manifolds}

\author[M.~Galli]{Matteo Galli} \address{Dipartimento di Matematica \\
Universit\`a di Bologna \\ Piazza di Porta S. Donato 5 \\ 40126 Bologna, Italy}
\email{matteo.galli8@unibo.it}

\author[M.~Ritor\'e]{Manuel Ritor\'e} \address{Departamento de
Geometr\'{\i}a y Topolog\'{\i}a \\
Universidad de Granada \\ E--18071 Granada \\ Espa\~na}
\email{ritore@ugr.es}

\date{\today}

\thanks{M. Galli has been supported by the People Programme (Marie Curie Actions) of the European Union's Seventh
Framework Programme FP7/2007-2013/ under REA grant agreement n. 607643. M. Ritor\'e has been supported by Mec-Feder MTM2010-21206-C02-01 and Mineco-Feder MTM2013-48371-C2-1-P research grants.}

\begin{abstract}
In this paper we consider surfaces of class $C^1$ with continuous prescribed mean curvature in a three-dimensional contact sub-Riemannian manifold and prove that their characteristic curves are of class $C^2$. This regularity result also holds for critical points of the sub-Riemannian perimeter under a volume constraint. All results are valid in the first Heisenberg group $\hh^1$.
\end{abstract}

\maketitle

\bibliographystyle{abbrv}

\thispagestyle{empty}

\section{Introduction}

Recently Cheng, Hwang and Yang \cite{MR2262784}, \cite{MR2481053}, have considered the functional
\begin{equation}
\label{eq:F}
\mathcal{F}(u)=\int_\Om |\nabla u+\vec{F}|+\int_\Om fu,
\end{equation}
on a domain $\Om\subset\rr^{2n}$, where $\vec{F}$ is a vector field and $f\in L^\infty(\Om)$. In case $\vec{F}(x,y)=(-y,x)$, the integral $\int_\Om |\nabla u+\vec{F}|$ is the sub-Riemannian area of the horizontal graph of the function $u$ in the Heisenberg group $\hh^n$. Among several interesting results, they proved in \cite[Thm.~A]{MR2481053} that, in case $n=1$, $u\in C^1(\Om)$ is an stationary point of $\mathcal{F}$ and $f\in C^0(\Om)$, the integral curves of the vector field $((\nabla u+\vec{F})/|\nabla u+\vec{F}|)^\bot$, defined in the set $|\nabla u+\vec{F}|\neq 0$, are of class $C^2$. The geometric meaning of their result is that the projection of the characteristic curves of the graph of $u$ are of class $C^2$. A stationary point $u$ of $\mathcal{F}$ satisfies weakly the prescribed mean curvature equation
\begin{equation}
\label{eq:pmc}
\divv\bigg(\frac{\nabla u+\vec{F}}{|\nabla u+\vec{F}|}\bigg)=f.
\end{equation}
Theorem~A in \cite{MR2481053} is well-known for $C^2$ minimizers and generalizes a previous result by Pauls \cite[Lemma~3.3]{MR2225631} for $H$-minimal surfaces with components of the horizontal Gauss map in the class $W^{1,1}$. For lipschitz continuous vanishing viscosity minimal graphs, it was proven by Capogna, Citti and Manfredini \cite[Cor.~1.6]{MR2583494}.

In order to extend this result to arbitrary surfaces, it is natural to replace $\mathcal{F}$   by the sub-Riemannian prescribed mean curvature functional
\begin{equation}
\label{eq:JOm}
\mathcal{J}(E,B)=P(E,B)+\int_{E\cap B} f,
\end{equation}
where $E$ is a set of locally finite sub-Riemannian perimeter in $\Om$, $P(E,B)$ is the relative sub-Riemannian perimeter of $E$ in a bounded open set $B\subset\Om$, and $f\in L^\infty(\Om)$. If $E\subset\hh^n$ is the subgraph of a function $t=u(x,y)$ in the Heisenberg group $\hh^n$, then $\mathcal{J}(E)$  coincides with \eqref{eq:F} taking $\vec{F}(x,y)=(-y,x)$.
The notion of sub-Riemannian perimeter used in sub-Riemannian geometry was first introduced by Capogna, Danielli and Garofalo \cite{MR1312686}  for Carnot-Carath\'eodory spaces. General properties and existence of sets with minimum perimeter were proved later by Garofalo and Nhieu \cite{MR1404326}. A rather complete theory of finite perimeter sets in the Heisenberg group $\hh^n$ following De Giorgi's original arguments was developed by Franchi, Serapioni and Serra-Cassano \cite{MR1871966}, and later extended to step $2$ Carnot groups \cite{MR1984849} by the same authors. The recent monograph \cite{MR2312336} provides a quite complete survey on recent progress on the subject.

We have defined the prescribed mean curvature functional following Massari \cite{MR0355766}, who considered minimizers of $\mathcal{J}$ for the Euclidean relative perimeter. He obtained existence and regularity results for this problem and observed that, in case $E$ is the subgraph of a Lipschitz function $u$ defined on an open bounded set $D\subset\rr^{n-1}$, the function $u$ satisfies weakly the prescribed mean curvature equation
\[
\divv\bigg(\frac{\nabla u}{\sqrt{1+|\nabla u|^2}}\bigg)(x)=f(x,u(x))
\]
for $x\in D$. In case $\ptl E\cap\Om$ is a hypersurface of class $C^2$ then the mean curvature of $\ptl E$ at a point $p\in\ptl E$ equals $g(p)$. See also Maggi \cite[pp.~139--140]{MR2976521}.


\mbox{}

The aim of this paper is to extend Cheng, Hwang and Yang's regularity result for characteristic curves \cite[Thm.~A]{MR2481053} from $C^1$ horizontal graphs satisfying weakly the mean curvature equation in the first Heisenberg group $\hh^1$ to surfaces of class $C^1$ with prescribed mean curvature in arbitrary three-dimensional contact sub-Riemannian manifolds. 

In this setting, the Euclidean perimeter is replaced by the sub-Riemannian one and the integral of the function $f$ is computed using Popp's measure \cite[\S~10.6]{MR1867362}, \cite{MR3108867}. The minimizing condition will be replaced by a stationary one. Our ambient space will be a three-dimensional contact manifold with a sub-Riemannian metric defined on its horizontal distribution. In particular, no assumption on the existence of a pseudo-hermitian structure is made. We shall prove in Theorem~\ref{th:main}
\begin{quotation}
Let $E\subset\Om$ be a set with $C^1$ boundary and prescribed mean curvature $f\in C^0(\Om)$ in a domain $\Om\subset M$ of a three-dimensional contact sub-Riemannian manifold. Then characteristic curves in $\ptl E$ are of class $C^2$.
\end{quotation}
We remark that  \cite[Thm.~A]{MR2481053} states that the projection of characteristic curves to the plane $t=0$ is of class $C^2$, but together with \cite[(2.22)]{MR2165405}  this implies that the characteristic curves themselves are $C^2$. We thank J.-H. Cheng for pointing out this fact. 

 While the proof of \cite[Thm.~A]{MR2481053} was based on the integral formula \cite[(2.3)]{MR2481053}, see also (3.7) in \cite[Remark~3.4]{gr-aim}, the proof of Theorem~\ref{th:main} is purely variational and follows by localizing the first variation of perimeter along a characteristic curve. A much weaker version of Theorem~\ref{th:main} was given in \cite[Thm.~3.5]{gr-aim}, where it was proven that the regular part of an area-stationary surface of class $C^1$ in the sub-Riemannian Heisenberg group $\hh^1$ is foliated by horizontal geodesics. Theorem~\ref{th:main} provides a new result even for the case of the first Heisenberg group $\hh^1$.

The regularity of characteristic curves proven in Theorem~\ref{th:main} allows us to define in Section~\ref{sec:mc} a mean curvature function $H$ in the regular part of $\ptl E$, that coincides with $f$. As a consequence of the definition of the mean curvature, we shall prove in Proposition~\ref{prop:C^k} that characteristic curves are of class $C^{k+2}$ in case $f$ is of class $C^k$ when restricted to a characteristic direction. This holds, e.g., when $f\in C^k(\Om)$ of $f\in C_\mathbb{H}^k(\Om)$, the space of functions with continuous horizontal derivatives of order $k$, $k\ge 1$. This class contains $C^1(\Om)$ when $k\ge 2$. Critical points of the perimeter, eventually under a volume constraint, and $C^1$ boundary, have constant prescribed mean curvature as shown in Section~\ref{sec:pmc}. Hence Theorem~\ref{th:main} applies to these sets and implies that the regular parts of their boundaries are foliated by $C^\infty$ characteristic curves, see Proposition~\ref{prop:critical}.


We have organized this paper into several sections. In the second one we provide the necessary background on contact sub-Riemannian manifolds and sets of finite perimeter, and we recall the first variation formula for $C^1$ surfaces following \cite{Gaphd}. In Section~\ref{sec:pmc} we introduce the definition of set of locally finite perimeter with prescribed mean curvature and prove that a set with $C^1$ boundary and area-stationary under a volume constraint has constant prescribed mean curvature. The main result, Theorem~\ref{th:main}, is proven in Section~\ref{sec:main}. The consequences on the mean curvature and higher regularity for characteristic curves will appear in Section~\ref{sec:mc}.


\section{Preliminaries}

\subsection{Contact sub-Riemannian manifolds}

In this paper we shall consider a $3$-dimensional $C^\infty$ manifold $M$ with contact form $\omega$ and a sub-Riemannian metric $g_\mathcal{H}$ defined on its \emph{horizontal distribution} $\mathcal{H}:=\text{ker}(\omega)$. By definition, $d\omega|_{\mathcal{H}}$ is non-degenerate. We shall refer to $(M,\omega,g_\mathcal{H})$ as a \emph{$3$-dimensional contact sub-Riemannian manifold}. It is well-known that $\omega\wedge d\omega$ is an orientation form in $M$. Since
\[
d\omega(X,Y)=X(\omega(Y))-Y(\omega(X))-\omega([X,Y]),
\]
the horizontal distribution $\mathcal{H}$ is completely non-integrable. The \emph{Reeb vector field} $T$ in $M$ is the only one satisfying
\begin{equation}
\label{eq:reeb}
\omega(T)=1,\qquad \mathcal{L}_T\omega=0,
\end{equation}
where $\mathcal{L}$ is the Lie derivative in $M$.

A canonical contact structure in Euclidean $3$-space $\rr^3$ with coordinates $(x,y,t)$ is given by the contact one-form $\omega_0:=dt+xdy-ydx$. The associated contact manifold is the Heisenberg group $\hh^1$. Darboux's Theorem \cite[Thm.~3.1]{MR1874240} (see also \cite{MR2979606}) implies that, given a point $p\in M$, there exists an open neighborhood $U$ of $p$ and a diffeomorphism $\phi_p$ from $U$ into an open set of $\rr^{3}$ satisfying $\phi_p^*\omega_0=\omega$. Such a local chart will be called a \emph{Darboux chart}. Composing the map $\phi_p$ with a contact transformation of $\hh^1$ also provides a Darboux chart. This implies we can prescribe the image of a point $p\in U$ and the image of a horizontal direction in $T_pM$.

The metric $g_\mathcal{H}$ can be extended to a Riemannian metric $g$ on $M$ by requiring $T$ to be a unit vector orthogonal to $\mathcal{H}$. The Levi-Civita connection associated to $g$ will be denoted by $D$. The integral curves of the Reeb vector field $T$ are \emph{geodesics} of the metric $g$. This property can be easily checked since condition $\mathcal{L}_T\omega=0$ in \eqref{eq:reeb} implies $\omega([T,X])=0$ for any\ $X\in\mathcal{H}$. Hence, for any horizontal vector field $X$, we have
\[
\escpr{X,D_TT}=-\escpr{D_TX,T}=-\escpr{D_XT,T}=0.
\] 
We trivially have $\escpr{T,D_TT}=0$, and so we get $D_TT=0$, as claimed.


The Riemannian volume element in $(M,g)$ will be denoted by $dM$. It coincides with Popp's measure \cite[\S~10.6]{MR1867362}, \cite{MR3108867}. The volume of a set $E\subset M$ with respect to the Riemannian metric $g$ will be denoted by $|E|$. 

\subsection{Torsion and the sub-Riemannian connection}

The following is taken from \cite[\S~3.1.2]{Gaphd}. In a contact sub-Riemannian manifold, we can decompose the endomorphism $X\in TM\rightarrow D_X T$ into its antisymmetric and symmetric parts, which we will denoted by $J$ and $\tau$, respectively,
\begin{equation}\label{eq:jtau}
\begin{split}
2\escpr{J(X),Y}&=\escpr{D_X T,Y}-\escpr{D_Y T,X},\\
2\escpr{\tau(X),Y}&=\escpr{D_X T,Y}+\escpr{D_Y T,X}.
\end{split}
\end{equation}
Observe that $J(X),\tau(X)\in\hhh$ for any vector field $X$, and that $J(T)=\tau(T)=0$. Also note that
\begin{equation}\label{eq:J}
2\escpr{J(X),Y}=-\escpr{[X,Y],T}, \qquad X,Y\in\mathcal{H}.
\end{equation}
We will call $\tau$ the \emph{$($contact$)$ sub-Riemannian torsion}. We note that our $J$ differs from the one defined in \cite[(2.4)]{MR3044134} by the constant $g([X,Y],T)$, but plays the same geometric role and can be easily generalized to higher dimensions, \cite[\S~3.1.2]{Gaphd}. 

Now we define the \emph{$($contact$)$ sub-Riemannian connection} $\nabla$ as the unique metric connection, \cite[eq.~(I.5.3)]{MR2229062}, with torsion tensor $\tor(X,Y)=\nabla_XY-\nabla_YX-[X,Y]$ given by
\begin{equation}\label{def:tor}
\tor(X,Y):=\escpr{X,T}\,\tau(Y)-\escpr{Y,T}\,\tau(X)+2\escpr{J(X),Y}\,T.
\end{equation}
From \eqref{def:tor} and Koszul formula for the connection $\nabla$ it follows that $T$ is a parallel vector field for the sub-Riemannian connection. In particular, their integral curves are geodesics for the connection $\nabla$.

If $X\in\mathcal{H}$, $p\in M$, and $X_p\neq 0$, then $J(X_p)\neq 0$: as $d\omega|_{\mathcal{H}}$ is non-degenerate, there exists $Y\in\mathcal{H}$ such that $d\omega_p(X_p,Y_p)\neq 0$. From \eqref{eq:jtau} we have $2\,g(J(X_p),Y_p)=-g([X,Y]_p,T_p)$, different from $0$ since $\omega_p([X,Y]_p)=-d\omega(X_p,Y_p)\neq 0$.

The standard orientation of $M$ is given by the $3$-form $\omega\wedge d\omega$. If $X_p$ is horizontal, then the basis $\{X_p,J(X_p),T_p\}$ is positively oriented. To check this, observe first that the sign of $(\omega\wedge d\omega)(X,J(X),T)$ equals the sign of $d\omega(X,J(X))$, and we have
\[
d\omega(X,J(X))=-\omega([X,J(X)])=-g([X,J(X)],T)=g(\tor(X,J(X)),T)=2\,g(J(X),J(X))>0.
\]

\subsection{Perimeter and $C^1$ surfaces}

A set $E\subset M$ has \emph{locally finite perimeter} if, for any bounded open set $B\subset M$, we have
\[
P(E,B):=\sup\bigg\{\int_{E\,\cap\,B}\divv U\,dM: U\ \text{horizontal},\ \text{supp}(U)\subset B, ||U||_\infty\le 1\bigg\}<+\infty.
\]
The quantity $P(E,B)$ is the \emph{relative perimeter} of $E$ in $B$.

Assuming $\Sg=\ptl E$ is a surface of class $C^1$, the relative perimeter of $E$ in a bounded open set $B\subset M$ coincides with the sub-Riemannian area of $\Sg\cap B$, given by
\begin{equation}
\label{eq:areaC1}
A(\Sg\cap B)=\int_{\Sg\,\cap\, B}|N_h|\,d\Sg.
\end{equation}
Here $N$ is the Riemannian unit normal to $\Sg$, $N_h$ is the horizontal projection of $N$ to the horizontal distribution, and $d\Sg$ is the Riemannian area measure, all computed with respect the Riemannian metric $g$, see \cite{MR2312336}. The quantity $|N_h|$ vanishes in the \emph{singular set} $\Sg_0\subset\Sg$ of points $p\in\Sg$ where the tangent space $T_p\Sg$ coincides with the horizontal distribution $\mathcal{H}_p$. The \emph{horizontal unit normal} at $p\in\Sg\setminus\Sg_0$ is defined by $(\nuh)_p:=(N_h)_p/|(N_h)_p|$. At every point $p\in\Sg\setminus\Sg_0$, the intersection $\mathcal{H}_p\cap T_p\Sg$ is one-dimensional and generated by the characteristic vector field $Z:=J(\nuh)/|J(\nuh)|$. The vector $S_p$ is defined for $p\in\Sg\setminus\Sg_0$ by $S_p:=g(N_p,T_p)\,(\nuh)_p-|(N_h)_p|\,T_p$. The tangent space $T_p\Sg$, $p\in\Sg\setminus\Sg_0$, is generated by $\{Z_p,S_p\}$.

\subsection{The first variation of the sub-Riemannian perimeter for $C^1$ surfaces}

Given a set $E$ with $C^1$ boundary, we can use the flow $\{\varphi_s\}_{s\in\rr}$ of a vector field $U$ with compact support in $B$ to produce a variation of $\Sg\cap B$. The Riemannian area formula gives the following expression of the sub-Riemannian area of $\Sg_s:=\varphi_s(\Sg\cap B)$,
\[
A(\Sg_s)=\int_\Sg |N^s_h|\,\Jac(\varphi_s)\,d\Sg,
\]
where $N^s$ is a unit normal to $\Sg_s$. Fix $p\in\Sg\setminus\Sg_0$ and the orthonormal basis $\{e_1,e_2\}=\{Z_p,S_p\}$ in $T_p\Sg$. We consider extensions $E_1$, $E_2$ of $Z_p$, $S_p$, respectively, along the integral curve of $U$ passing through $p$. The vector fields $E_1(s)$, $E_2(s)$ are invariant under the flow of $U$ and generate the tangent plane to $\Sg_s$ at the point $\varphi_s(p)$. The vector $(E_1\times E_2)/|(E_1\times E_2)|$ is normal to $\Sg_s$. Here $\times$ denotes the cross product with respect to a volume form $\eta$ for the metric $g$ inducing the same orientation as $\omega\wedge d\omega$, i.e. $g(w,u\times v)=\eta(w,u,v)$. It is easy to check that $|(E_1\times E_2)|(s)=\Jac(\varphi_s)(p)$, and that
\[
V(p,s):=(E_1\times E_2)_h(s)=\big(\escpr{E_1,T}\,(T\times E_2)-\escpr{E_2,T}\,(E_1\times T)\big)(s).
\]
Hence
\[
A(\Sg_s)=\int_\Sg |V(p,s)|\,d\Sg(p),
\]
and we get
\[
\frac{d}{ds}\bigg|_{s=0} |V(s,p)|=\frac{\escpr{\nabla_{U_p} V,V_p}}{|V_p|}.
\]
Since $\{(\nuh)_p,Z_p,T_p\}$ is positively oriented, observe that $V_p=|(N_h)_p|\,(\nu_h)_p$. On the other hand,
\begin{align*}
\nabla_{U_p} V=\escpr{\nabla_{U_p} E_1,T_p}\,(T_p\times (E_2)_p)-\escpr{\nabla_{U_p}E_2,T_p}\,&((E_1)_p\times T_p)
\\
&-\escpr{(E_2)_p,T_p}\,(\nabla_{U_p} E_1\times T_p),
\end{align*}
and so
\begin{equation*}
\frac{\escpr{\nabla_{U_p}V,V_p}}{|V_p|}=-\escpr{\nabla_{U_p}E_2,T_p}+|(N_h)_p|\,\escpr{\nabla_{U_p} E_1\times T_p,(\nuh)_p}.
\end{equation*}
Since
\begin{align*}
\escpr{\nabla_{U_p}E_2,T_p}&=\escpr{\nabla_{(E_2)_p}U+\tor(U_p,(E_2)_p),T_p}
\\
&=S_p(\escpr{U,T})+\escpr{\tor(U_p,S_p),T_p}
\\
&=S_p(\escpr{U,T})+2\,\escpr{J(U_p),S_p},
\end{align*}
and
\begin{align*}
\escpr{\nabla_{U_p} E_1\times T_p,(\nuh)_p}&=\escpr{(\nabla_{(E_1)_p}U+\tor(U_p,(E_1)_p))\times T_p,(\nuh)_p}
\\
&=\eta((\nuh)_p,\nabla_{(E_1)_p}U+\tor(U_p,(E_1)_p),T_p)
\\
&=+\escpr{\nabla_{Z_p}U+\tor(U_p,Z_p),Z_p}
\\
&=+\escpr{\nabla_{Z_p}U,Z_p}+\escpr{U_p,T_p}\,\escpr{\tau(Z_p),Z_p},
\end{align*}
we conclude that the first variation of the sub-Riemannian perimeter is given by
\begin{equation}
\label{eq:1starea}
\begin{split}
\frac{d}{ds}\bigg|_{s=0} A(\Sg_s)=\int_{\Sg\,\cap\,B}\big\{-S(g(U,T))-2\,g(J(U),S)&+|N_h|\,g(\nabla_ZU,Z)
\\
&+|N_h|\,\escpr{U,T}\,\escpr{\tau(Z),Z}\big\}\,d\Sg.
\end{split}
\end{equation}
This formula was obtained in \cite[Lemma~3.4]{MR3044134}.

\section{Sets with prescribed mean curvature}
\label{sec:pmc}

The reader is referred to \cite[(12.32) and Remark~17.11]{MR2976521} for background and references in the Euclidean case. Consider a domain $\Om\subset M$, and a function $f:\Om\to\rr$. We shall say that a set of locally finite perimeter $E\subset\Om$ has \emph{prescribed mean curvature $f$ on $\Om$} if, for any bounded open set $B\subset\Om$, $E$ is a critical point of the functional
\begin{equation}
\label{eq:a-hv}
P(E,B)-\int_{E\cap B} f,
\end{equation}
where $P(E,B)$ is the relative perimeter of $E$ in $B$, and the integral on $E\cap B$ is computed with respect to the canonical Popp's measure on $M$, see \cite{MR1867362} and \cite{MR3108867}. The admissible variations for this problem are the flows induced by vector fields with compact support in $B$.

If $\Sg=\ptl E$ is a surface of class $C^1$ in $\Om$, then $\Sg$ has prescribed mean curvature $f$ if it is a critical point of the functional
\begin{equation}\label{eq:prescribedfunctional}
A(\Sg\cap B) -\int_{E\cap B} f,
\end{equation}
for any bounded open set $B\subset\Om$.

If $E$ is a critical point of the relative perimeter $P(E,B)$ in any bounded open set $B\subset\Om$, then $E$ has zero or vanishing prescribed mean curvature.

Assume now that $E\subset\Om$ is a set of locally finite perimeter with $C^1$ boundary $\Sg$, and that $E$ is a \emph{critical point of the perimeter under a volume constraint}. This means $(d/ds)_{s=0} A(\varphi_s(\Sg\cap B))=0$ for any flow associated to a vector field with compact support in $\Om$ satisfying $(d/ds)_{s=0} |\varphi_s(E\cap B)|=0$. If the perimeter of $E$ in $\Om$ is positive, then there exists a (horizontal) vector field $U_0$ with compact support in $\Om$ so that $\int_{E\cap\Om}\divv U_0\,dM>0$. By the Divergence Theorem,
\[
\int_{\Sg\cap\Om} g(U_0,N)\,d\Sg\neq 0,
\]
where $N$ is the outer normal to $E$. Let $\{\psi_s\}_{s\in\rr}$ be the flow associated to the vector field $U_0$ and define
\begin{equation}
\label{eq:defH0}
H_0:=\frac{d/ds\big|_{s=0}\, A(\psi_s(\Sg))}{d/ds\big|_{s=0}|\psi_s(E)|}.
\end{equation}
Let $B\subset \Om$ be a bounded open subset and $W$ a vector field with compact support in $B$ and associated flow $\{\varphi_s\}_{s\in\rr}$. Choose $\la\in\rr$ so that $W-\la U_0$ satisfies
\[
\frac{d}{ds}\bigg|_{s=0}|\varphi_s(E)|-\la\,\frac{d}{ds}\bigg|_{s=0}|\psi_s(E)|=\int_{\Sg} g(W-\la U_0,N)\,d\Sg=0.
\]
Then the flow associated to $W-\la U_0$ preserves the volume of $E\cap(B\cup B_0)$, where $B_0\subset\Om$ is a bounded open set containing $\text{supp}(U_0)$. Let $Q(U)$ be the integral expression in \eqref{eq:1starea}. From our hypothesis and the linearity of \eqref{eq:1starea}, $Q(W-\la U_0)=0$. Hence $Q(W)=\la\,Q(U_0)$. From \eqref{eq:defH0} we get
\[
Q(W)=\la Q(U_0)=\la H_0\frac{d}{ds}\bigg|_{s=0}|\psi_s(E)|=H_0\frac{d}{ds}\bigg|_{s=0}|\varphi_s(E)|,
\]
and so $E$ has (constant) prescribed mean curvature $H_0$.



\section{Main result}
\label{sec:main}

In this Section we shall prove our main result

\begin{theorem}
\label{th:main}
Let $M$ be a $3$-dimensional contact sub-Riemannian manifold, $\Om\subset M$ a domain, and $E\subset\Om$ a set of prescribed mean curvature $f\in C^0(\Om)$ with $C^1$ boundary $\Sg$. Then the characteristic curves in $\Sg$ are of class $C^2$.
\end{theorem}

\begin{proof}
Given any point $p\in\Sg\setminus\Sg_0$, consider a Darboux chart $(U_p,\phi_p)$ such that $\phi_p(p)=0$. The metric $g_{\mathcal{H}}$ can be described in this local chart by the matrix of smooth functions
\[
G=\begin{pmatrix}
g_{11} & g_{12}
\\
g_{21} & g_{22}
\end{pmatrix}
=
\begin{pmatrix}
g(X,X) & g(X,Y)
\\
g(Y,X) & g(Y,Y)
\end{pmatrix}.
\]
After a Euclidean rotation around the $t$-axis, which is a contact transformation in $\hh^1$ \cite[p.~640]{MR2435652}, we may assume there exists an open neighborhood $B\cap\Sg$ of $p\in \Sg\setminus\Sg_0$, where $B\subset\hh^1$ is an open set containing $p$, so that $B\cap\Sg$ is the intrinsic graph $G_u$ of a $C^1$ function $u:D\to\rr$ defined on a domain $D$ in the vertical plane $y=0$. We can also assume that $E\cap B$ is the subgraph of $u$. The graph $G_u$ can be parameterized by the map $f_u:D\to\rr^3$ defined by
\[
f_u(x,t):=(x,u(x,t),t-xu(x,t)),\qquad (x,t)\in D.
\]
The tangent plane to any point in $G_u$ is generated by the vectors
\begin{align*}
\tfrac{\ptl}{\ptl x}&\mapsto (1,u_x,-u-xu_x)=X+u_xY-2uT,
\\
\tfrac{\ptl}{\ptl t}&\mapsto (0,u_t,1-xu_t)=u_tY+T,
\end{align*}
and so the characteristic direction is given by $Z=\widetilde{Z}/|\widetilde{Z}|$, where
\[
\widetilde{Z}=X+(u_x+2uu_t)Y.
\]

If $\ga(s)=(x(s),t(s))$ is a $C^1$ curve in $D$, then
\[
\Ga(s)=(x(s),u(x(s),t(s)),t(s)-x(s)u(x(s),t(s)))\subset G_u
\]
is also $C^1$, and so
\[
\Ga'(s)=x'\,(X+u_xY-2uT)+t'\,(u_tY+T)=x'X+(x'u_x+t'u_t)Y+(t'-2ux')T.
\]
In particular, horizontal curves in $G_u$ satisfy the ordinary differential equation $t'=2ux'$. Since $u\in C^1(D)$, we have uniqueness of characteristic curves through any given point in $G_u$.

A unit normal vector to $\Sg$ is given by $\widetilde{N}/|\widetilde{N}|$, where
\[
\widetilde{N}=(X+u_xY-2uT)\times (u_tY+T).
\]
Here $\times$ is the cross product with respect to the Riemannian metric $g$ and a given volume form $\eta$ chosen so that $\eta(X,Y,T)>0$. Hence $g(w,u\times v)=\eta(w,u,v)$. If $\{e_1,e_2,e_3\}$ is an orthonormal basis so that $\eta(e_1,e_2,e_3)=1$ and $A$ is the matrix whose columns are the coordinates of $X$, $Y$, $T$ in the basis $\{e_1,e_2,e_3\}$, then $\eta(X,Y,T)=\det(A)$. On the other hand, as
\[
A^tA=\begin{pmatrix}
g_{11} & g_{12} & 0
\\
g_{21} & g_{22} & 0
\\
0 & 0 & 1
\end{pmatrix},
\]
we get $\det(A)^2=\det(G)$. Since $\det(A)>0$ we obtain $\det(A)=\det(G)^{1/2}$ and so
\[
\eta(X,Y,T)=\det(G)^{1/2}.
\]
Let $E_1=X+u_xY-2uT$, $E_2=u_tY+T$. We compute the scalar product of $\widetilde{N}=E_1\times E_2$ with $X$, $Y$, $T$ to obtain
\begin{align*}
g(X,E_1\times E_2)&=\eta(X,E_1,E_2)=\det
{\arraycolsep=0.3\arraycolsep\begin{pmatrix}
1 & 1 & 0
\\
0 & u_x & u_t
\\
0 & -2u & 1
\end{pmatrix}}
\,\eta(X,Y,T)=(u_x+2uu_t)\det(G)^{1/2}.
\\
g(Y,E_1\times E_2)&=\eta(Y,E_1,E_2)=\det
{\arraycolsep=0.3\arraycolsep\begin{pmatrix}
0 & 1 & 0
\\
1 & u_x & u_t
\\
0 & -2u & 1
\end{pmatrix}}
\,\eta(X,Y,T)=-\det(G)^{1/2}.
\\
g(T,E_1\times E_2)&=\eta(T,E_1,E_2)=\det
{\arraycolsep=0.3\arraycolsep\begin{pmatrix}
0 & 1 & 0
\\
0 & u_x & u_t
\\
1 & -2u & 1
\end{pmatrix}}
\,\eta(X,Y,T)=u_t \det(G)^{1/2}.
\end{align*}
Since $g(Y,E_1\times E_2)<0$ and $E\cap B$ is the subgraph of $u$, the vector field $E_1\times E_2$ points into the interior of $E$. If $\widetilde{N}=E_1\times E_2=\alpha X+\beta Y+\gamma T$, then
\[
\begin{pmatrix}
g_{11} & g_{12} & 0
\\
g_{21} & g_{22} & 0
\\
0 & 0 & 1
\end{pmatrix}
\begin{pmatrix}
\alpha \\ \beta \\ \gamma
\end{pmatrix}=\det(G)^{1/2}
\begin{pmatrix}
u_x+2uu_t \\ -1 \\ u_t
\end{pmatrix},
\]
whence
\begin{equation}
\label{eq:Ntilde}
\begin{split}
\begin{pmatrix} \alpha \\ \beta \end{pmatrix}&=\det(G)^{1/2}\, G^{-1}
\begin{pmatrix} u_x+2uu_t \\ -1 \end{pmatrix},
\\
\gamma&=\det(G)^{1/2} u_t.
\end{split}
\end{equation}

Let us compute now the sub-Riemmanian area of the intrinsic graph $G_u$. It is easy to check that $d\Sg=|E_1\times E_2|\,dxdt$, i.e., that $\Jac(f_u)=|E_1\times E_2|$. Since $|N_h|=|(E_1\times E_2)_h|/|E_1\times E_2|$, then using \eqref{eq:Ntilde} and the explicit expression of the inverse matrix $G^{-1}$ we get
\begin{align*}
|N_h|\,\Jac(f_u)&=|(E_1\times E_2)_h|=\bigg(\big(\alpha\ \ \beta\big)\, G\, 
\begin{pmatrix} \alpha \\ \beta \end{pmatrix}\bigg)^{\frac{1}{2}}
\\
&=\big((g_{22}\circ f_u)(u_x+2uu_t)^2+2\,(g_{12}\circ f_u)(u_x+2uu_t)+(g_{11}\circ f_u)\big)^{1/2}.
\end{align*}
Finally, from \eqref{eq:areaC1} we obtain
\begin{equation}
\label{eq:areagraph}
A(G_u)=\int_D \big(g_{22}(u_x+2uu_t)^2+2g_{12}(u_x+2uu_t)+g_{11}\big)^{1/2}dxdt,
\end{equation}
where, by abuse of notation, we have written $g_{ij}$ instead of the cumbersome notation $(g_{ij}\circ f_u)$.

Now we consider variations of $G_u$ by graphs of the form $s\mapsto u+sv$, where $v\in C_0^\infty(D)$ and $s$ is a real parameter close to $0$. This variation is obtained by applying the flow associated to the vector field $\tilde{v}Y$ to the graph $G_u$. The function $\tilde{v}$ is obtained by extending $v$ to be constant along the integral curves of the vector field $Y$, and multiplying by an appropriate function with compact support equal to $1$ in a neighborhood of $\Sg$.

When $F$ is a function of $(x,y,t)$, we have
\[
\frac{d}{ds}\bigg|_{s=0} (F\circ f_{u+sv})(x,t)=\bigg(\frac{\ptl F}{\ptl y}-x\frac{\ptl F}{\ptl t}\bigg)_{f_u(x,t)}\,v(x,t)=Y_{f_u(x,t)} (F)\, v(x,t).
\]
So we get 
\begin{equation*}
\frac{d}{ds}\bigg|_{s=0} A(G_{u+sv})
=\int_D \big(K_1v+M\,(v_x+2uv_t+2vu_t)\big)\,dxdt,
\end{equation*}
where the functions $K_1$ and $M$ are given by
\[
K_1=\frac{1}{2}\frac{Y(g_{22})(u_x+2uu_t)^2+2 Y(g_{12})(u_x+2uu_t)+Y(g_{11})}{(g_{22}(u_x+2uu_t)^2+2g_{12}(u_x+2uu_t)+g_{11})^{1/2}},
\]
and
\[
M=\frac{g_{22}(u_x+2uu_t)+g_{12}}{(g_{22}(u_x+2uu_t)^2+2g_{12}(u_x+2uu_t)+g_{11})^{1/2}}.
\]
Observe that the functions $K_1$ and $M$ are continuous. Since
\[
Z=\frac{X+(u_x+2uu_t)\,Y}{(g_{22}(u_x+2uu_t)^2+2g_{12}(u_x+2uu_t)+g_{11})^{1/2}},
\]
the function $M$ coincides with $g(Z,Y)\circ f_u$. A straightforward computation implies
\[
1=|Z|^2=\det(G)^{-1}\big(g_{22}\escpr{Z,X}^2-2g_{12}\escpr{Z,X}\escpr{Z,Y}+g_{11}\escpr{Z,Y}^2\big)
\]
and so
\begin{equation}
\label{eq:ZX}
\escpr{Z,X}=\frac{g_{12}g(Z,y)\pm (\det(G)(g_{22}-\escpr{Z,Y}^2))^{1/2}}{g_{22}}.
\end{equation}
By Schwarz's inequality $\escpr{Z,Y}^2\le \escpr{Y,Y}=g_{22}$. Inequality is strict since otherwise $Y$ and $Z$ would be collinear. Hence $\escpr{Z,X}$ has the same regularity as $\escpr{Z,Y}$ by \eqref{eq:ZX}.

The subgraph of $u$ can be parameterized by the map $(x,t,s)\to (x,s,t-xs)$. The Jacobian of this map is easily seen to be equal to $\det(G)$. Hence
\[
\frac{d}{ds}\bigg|_{s=0} \int_{\text{subgraph}\,G_{u+sv}} f=\int_D f\,\det(G)\,v\,dxdt.
\]
If $\Sg$ has prescribed mean curvature $f$, this implies
\begin{equation}\label{eq:1stvariation smooth test function}
\int_D \big(Kv+M\,(v_x+2uv_t+2vu_t\big)\,dxdt=0,
\end{equation}
for any $v\in C_0^\infty(D)$, where the \emph{continuous} function $K$ is given by $K=K_1-f\,\det(G)$. By Remark \ref{remark:continuos approximation argument} below, \eqref{eq:1stvariation smooth test function} also holds  for any $v\in C^0_0(D)$ for which $v_x+2uv_t$ exists and it  is continuous.

Now we proceed as in the proof of Theorem~3.5 in \cite{gr-aim}. Assume the point $p\in G_u$ corresponds to the point $(a,b)$ in the $xt$-plane. The  curve $s\mapsto (s,t(s))$ is (a reparameterization of the projection of) a characteristic curve if and only if the function $t(s)$ satisfies the ordinary differential equation $t'(s)=u(s,t(s))$. For $\eps$ small enough, we consider the solution $t_\eps$ of equation $t_\eps'(s)=2u(s,t_\eps(s))$ with initial condition $t_\eps(a)=b+\eps$, and define $\ga_\eps(s):=(s,t_\eps(s))$, with $\ga=\ga_0$. We may assume that, for small enough $\eps$, the functions $t_\eps$ are defined in the interval $[a-r,a+r]$ for some $r>0$. The function $\ptl t_\eps/\ptl\eps$ satisfies
\begin{equation}
\label{eq:dteps}
\bigg(\frac{\ptl t_\eps}{\ptl\eps}\bigg)'(s)=2u_t(s,t_\eps(s))\,\bigg(\frac{\ptl t_\eps}{\ptl\eps}\bigg)(s),\qquad \frac{\ptl t_\eps}{\ptl\eps}(a)=1.
\end{equation}
where $'$ is the derivative with respect to the parameter $s$.

We consider the parameterization
\[
F(\xi,\eps):= (\xi,t_\eps(\xi))=(s,t)
\]
near the characteristic curve through $(a,b)$. The jacobian of this parameterization is given by
\[
\det\begin{pmatrix}
1 & t'_\eps
\\
0 & \ptl t_\eps/\ptl \eps
\end{pmatrix}=\frac{\ptl t_\eps}{\ptl\eps},
\]
which is positive because of the choice of initial condition for $t_\eps$ and the fact that the curves $\ga_\eps(s)$ foliate a neighborhood of $(a,b)$. Any function $\varphi$ can be considered as a function of the variables $(\xi,\eps)$ by making $\tilde{\varphi}(\xi,\eps):=\varphi(\xi,t_\eps(\xi))$. Changing variables, and assuming the support of $\varphi$ is contained in a sufficiently small neighborhood of $(a,b)$, we can express the integral \eqref{eq:areagraph} as
\[
\int_{I}\bigg\{\int_{a-r}^{a+r}\bigg(K\tilde{\varphi}+M\,\bigg(\frac{\ptl\tilde{\varphi}}{\ptl\xi}+2\tilde{\varphi} \tilde{u}_t\bigg)\bigg)\,\frac{\ptl t_\eps}{\ptl\eps}\,d\xi\bigg\}\,d\eps,
\]
where $I$ is a small interval containing $0$. Instead of $\tilde{\varphi}$, we can consider the function $\tilde{\varphi} h /( t_{\eps+h}-t_\eps)$, where $h$ is a sufficiently small real parameter. We get that 
\[
\frac{ \ptl}{\ptl\xi}\bigg(\frac{h\, \tilde{\varphi}}{ t_{\eps+h}-t_\eps}\bigg)=\frac{\ptl\tilde{\varphi}}{\ptl\xi}\cdot \frac{h}{t_{\eps+h}-t_\eps}-2\tilde{\varphi}\cdot \frac{\tilde{u}(\xi, \eps+h)-\tilde{u}(\xi,\eps)}{t_{\eps+h}-t_\eps}\cdot \frac{h}{t_{\eps+h}-t_\eps}
\]
tends to 
\[
\frac{\ptl\tilde{\varphi}/\ptl\xi}{\ptl t_\eps/\ptl\eps}-\frac{2\tilde{\varphi}\tilde{u}_t}{\ptl t_\eps/\ptl\eps},
\]
when $h\rightarrow 0$. So using that $G_u$ is area-stationary we have that
\[
\int_{I}\bigg\{\int_{a-r}^{a+r} \frac{h}{t_{\eps+h}-t_\eps}\bigg(K\,  \tilde{\varphi}+M\,\bigg(\frac{\ptl\tilde{\varphi}}{\ptl\xi}\cdot +2  \tilde{\varphi}\cdot \bigg(\tilde{u}_t-\frac{\tilde{u}(\xi, \eps+h)-\tilde{u}(\xi,\eps)}{t_{\eps+h}-t_\eps}\bigg)\bigg)\bigg)\,\frac{\ptl t_\eps}{\ptl\eps}\,d\xi\bigg\}\,d\eps
\]
vanishes. Furthermore, letting $h\rightarrow 0$ we conclude
\[
\int_{I}\bigg\{\int_{a-r}^{a+r}\bigg(K\tilde{\varphi}+M\,\frac{\ptl\tilde{\varphi}}{\ptl\xi}\bigg)\,d\xi\bigg\}\,d\eps=0.
\]


Let now $\eta:\rr\to\rr$ be a positive function with compact support in the interval $I$ and consider the family $\eta_\rho(x):=\rho^{-1}\eta(x/\rho)$. Inserting a test function of the form $\eta_\rho(\eps)\psi(\xi)$, where $\psi$ is a $C^\infty$ function with compact support in $(a-r,a+r)$, making $\rho\to 0$, and using that $G_u$ is area-stationary we obtain
\[
\int_{a-r}^{a+r} \big(K(0,\xi)\,\psi(\xi)+M(0,\xi)\,\psi'(\xi)\big)\,d\xi=0
\]
for any $\psi\in C^\infty_0((a-r,a+r))$. By Lemma~\ref{lem:c1}, the function $M(0,\xi)$, which is the restriction of $\escpr{Z,Y}$ to the characteristic curve, is a $C^1$ function on the curve. By equation \eqref{eq:ZX}, the restriction of $\escpr{Z,X}$ to the characteristic curve is also $C^1$. This proves that horizontal curves are of class $C^2$.
\end{proof}

\begin{lemma}
\label{lem:c1}
Let $I\subset\rr$ be an open interval, $k$, $m\in C^0(I)$, and $K\in C^1(I)$ be a primitive of $k$. Assume
\begin{equation}
\label{eq:inteq}
\int_I k\psi+m\psi'=0,
\end{equation}
for any $\psi\in C_0^\infty(I)$. Then the function $-K+m$ is constant on $I$. In particular, $m\in C^1(I)$.
\end{lemma}

\begin{proof}
Since $(K\psi)'=k\psi+K\psi'$, integrating by parts we see that \eqref{eq:inteq} is equivalent to
\[
\int_I \big(-K+m\big)\,\psi'=0,
\]
for any $\psi\in C^\infty_0(I)$. This implies that $-K+m$ is a constant function on $I$.
\end{proof}

\begin{remark}\label{remark:continuos approximation argument} Let us check that $\eqref{eq:1stvariation smooth test function}$ holds for any $w\in C^0_0(D)$ such that $w_x+2uw_t$ exists and is continuous. 
Let us consider a sequence $w_j\in C_0^\infty(D)$, where $w_j=\rho_j \ast w$, and $\rho_j$ denote the standard mollifiers. We have that $w_j$ converges to $w$ and that $(w_j)_x+2u(w_j)_t$ converges to $w_x+2uw_t$ uniformly on compact subsets of $D$, for $j\rightarrow \infty$.  We conclude
\[
0=\lim\limits_{j\rightarrow\infty} \int_D \big(K(w_j)+M\,((w_j)_x+2u(w_j)_t+2w_ju_t\big)\,dxdt= \int_D \big(Kw+M\,(w_x+2uw_t+2wu_t\big)\,dxdt,
\]
thus proving the claim.
\end{remark}

\begin{remark}
In case $M$ is the Heisenberg group $\hh^1$, $G$ is the identity matrix and the expression for the sub-Riemannian area of the graph $G_u$ given in \eqref{eq:areagraph} reads
\[
A(G_u)=\int_D ((u_x+2uu_t)^2+1)^{1/2}\,dxdt,
\]
a well-known formula obtained in \cite{MR2223801}.
\end{remark}

\section{The mean curvature for $C^1$ surfaces}
\label{sec:mc}

Given a surface $\Sg\subset M$ of class $C^1$ such that the vector fields $Z$ and $\nu_h$ are of class $C^1$ along the characteristic curves in $\Sg\setminus\Sg_0$, we define the \emph{mean curvature} of $\Sg$ at $p\in \Sg\setminus\Sg_0$ by
\begin{equation}
\label{eq:Hstrongsense}
\big(\divv_\Sg^h(\nuh)\big)(p):= -g(\n_Z \nu_h,Z)(p).
\end{equation}
This is the standard definition of mean curvature for $C^2$ surfaces, see e.g. \cite[(3.8)]{MR3044134} and the references there. The mean curvature is usually denoted by $H$. The mean curvature depends on the choice of $\nuh$. In case $\Sg$ is the boundary of a set $E$, we shall always choose $N$ as the inner normal and $\nuh=N_h/|N_h|$.

Using the regularity  Theorem~\ref{th:main} we get

\begin{proposition}
\label{prop:1stareaStrong}
Let $\Omega\subset M$ be a domain and $E\subset \Omega$ a set of prescribed mean curvature $f\in C^0(\Omega)$ with $C^1$ boundary $\Sg$ with $H\in L^1_{loc}(\Sg)$. Then the first variation of the functional \eqref{eq:prescribedfunctional} induced by a vector field $U\in C^1_0(\Omega)$ is given by 
\begin{equation}
\label{eq:1stvariationprescribedfunctional}
\begin{split}
\int_{\Sg}H\, g(U,N) \,d\Sg-\int_\Sg f \,g(U,N)\,d\Sg.
\end{split}
\end{equation}
\end{proposition}

\begin{proof}
We first observe that formula
\[
\frac{d}{ds}\bigg|_{s=0} A(\Sg_s)=-\int_{\Sg}H\, g(U,N) \,d\Sg
\] 
can be proved as in \cite[Prop.~6.3]{MR3044134}, see also \cite[Remark~6.4]{MR3044134}. On the other hand, it is well-known that  
\[
\frac{d}{ds}\bigg|_{s=0} \bigg(\int_{\phi_s(\Omega)} f \bigg)= -\int_\Sg f \,g(U,N)\, d\Sg,
\]
see e.g. \cite[17.8]{MR2976521}.
\end{proof}

Then we have that the mean curvature $H$ defined in \eqref{eq:Hstrongsense} coincides with the prescribed mean curvature $f$

\begin{corollary} Let $E\subset\Om$ be a set of prescribed mean curvature $f\in C^0(\Omega)$ with $C^1$ boundary $\Sg$ in a domain $\Om\subset M$. Assume $H\in L^1_{loc}(\Sg)$. Then $H(p)=f(p)$ for any $p\in\Sg\setminus\Sg_0$. 
\end{corollary}

Finally we can improve the regularity of Theorem~\ref{th:main} assuming the mean curvature function is more regular. This result is specially useful if we assume that higher order horizontal derivatives of the function $f$ exist and are continuous.

\begin{proposition}
\label{prop:C^k}
Let $E\subset\Om$ be a set of prescribed mean curvature $f\in C^0(\Omega)$ with $C^1$~boundary $\Sg$ in a domain $\Om\subset M$. Assume that $f$ is also $C^k$ in the $Z$-direction, $k\ge 1$. Then the characteristic curves of $\Sg$ are of class $C^{k+2}$ in $\Sg\setminus\Sg_0$. 
\end{proposition}

\begin{proof} Since we have $\n_Z Z= H \nu_h$, we can write
\[
\n_Z(\n_Z Z)= Z(H) \nu_h+H \n_Z \nu_h= Z(H)\nu_h-H^2 Z.
\]
Iterating the procedure we obtain the statement. 
\end{proof}

In particular, this result holds when $f\in C_\mathbb{H}^1(\Om)$, i.e., when $f$ has horizontal derivatives of class $(k-1)$, see \cite{MR1871966}.

An important particular case is that of critical points of perimeter, possibly under a volume constraint. Assuming $C^1$ regularity of the boundary, these sets are known to have constant prescribed mean curvature from the discussion in Section~\ref{sec:pmc}. From Proposition~\ref{prop:C^k}, we immediately obtain

\begin{proposition}
\label{prop:critical}
Let $E\subset\Om$ be either a critical point of the sub-Riemannian perimeter or a critical point of the sub-Riemannian perimeter under a volume constraint. If $E$ has $C^1$ boundary, then the regular part of $\ptl E$ is foliated by $C^\infty$ characteristic curves.
\end{proposition}

\bibliography{a-hv}

\begin{thebibliography}{10}

\bibitem{MR2223801}
L.~Ambrosio, F.~Serra~Cassano, and D.~Vittone.
\newblock Intrinsic regular hypersurfaces in {H}eisenberg groups.
\newblock {\em J. Geom. Anal.}, 16(2):187--232, 2006.

\bibitem{MR3108867}
D.~Barilari and L.~Rizzi.
\newblock A formula for {P}opp's volume in sub-{R}iemannian geometry.
\newblock {\em Anal. Geom. Metr. Spaces}, 1:42--57, 2013.

\bibitem{MR1874240}
D.~E. Blair.
\newblock {\em Riemannian geometry of contact and symplectic manifolds}, volume
  203 of {\em Progress in Mathematics}.
\newblock Birkh\"auser Boston, Inc., Boston, MA, 2002.

\bibitem{MR2583494}
L.~Capogna, G.~Citti, and M.~Manfredini.
\newblock Regularity of non-characteristic minimal graphs in the {H}eisenberg
  group {$\Bbb H^1$}.
\newblock {\em Indiana Univ. Math. J.}, 58(5):2115--2160, 2009.

\bibitem{MR1312686}
L.~Capogna, D.~Danielli, and N.~Garofalo.
\newblock The geometric {S}obolev embedding for vector fields and the
  isoperimetric inequality.
\newblock {\em Comm. Anal. Geom.}, 2(2):203--215, 1994.

\bibitem{MR2312336}
L.~Capogna, D.~Danielli, S.~D. Pauls, and J.~T. Tyson.
\newblock {\em An introduction to the {H}eisenberg group and the
  sub-{R}iemannian isoperimetric problem}, volume 259 of {\em Progress in
  Mathematics}.
\newblock Birkh\"auser Verlag, Basel, 2007.

\bibitem{MR2229062}
I.~Chavel.
\newblock {\em Riemannian geometry}, volume~98 of {\em Cambridge Studies in
  Advanced Mathematics}.
\newblock Cambridge University Press, Cambridge, second edition, 2006.
\newblock A modern introduction.

\bibitem{MR2165405}
J.-H. Cheng, J.-F. Hwang, A.~Malchiodi, and P.~Yang.
\newblock {Minimal surfaces in pseudohermitian geometry}.
\newblock {\em Ann. Sc. Norm. Super. Pisa Cl. Sci. (5)}, 4(1):129--177, 2005.

\bibitem{MR2262784}
J.-H. Cheng, J.-F. Hwang, and P.~Yang.
\newblock {Existence and uniqueness for {$p$}-area minimizers in the
  {H}eisenberg group}.
\newblock {\em Math. Ann.}, 337(2):253--293, 2007.

\bibitem{MR2481053}
J.-H. Cheng, J.-F. Hwang, and P.~Yang.
\newblock {Regularity of {$C^1$} smooth surfaces with prescribed {$p$}-mean
  curvature in the {H}eisenberg group}.
\newblock {\em Math. Ann.}, 344(1):1--35, 2009.

\bibitem{MR1871966}
B.~Franchi, R.~Serapioni, and F.~{Serra Cassano}.
\newblock {Rectifiability and perimeter in the {H}eisenberg group}.
\newblock {\em Math. Ann.}, 321(3):479--531, 2001.

\bibitem{MR1984849}
B.~Franchi, R.~Serapioni, and F.~Serra~Cassano.
\newblock On the structure of finite perimeter sets in step 2 {C}arnot groups.
\newblock {\em J. Geom. Anal.}, 13(3):421--466, 2003.

\bibitem{Gaphd}
M.~Galli.
\newblock {\em Area-stationary surfaces in contact sub-{R}iemannian manifolds}.
\newblock PhD thesis, Universidad de Granada, Available at
  http://hera.ugr.es/tesisugr/21013020.pdf, 2012.

\bibitem{MR3044134}
M.~Galli.
\newblock First and second variation formulae for the sub-{R}iemannian area in
  three-dimensional pseudo-{H}ermitian manifolds.
\newblock {\em Calc. Var. Partial Differential Equations}, 47(1-2):117--157,
  2013.

\bibitem{MR2979606}
M.~Galli and M.~Ritor{\'e}.
\newblock Existence of isoperimetric regions in contact sub-{R}iemannian
  manifolds.
\newblock {\em J. Math. Anal. Appl.}, 397(2):697--714, 2013.

\bibitem{gr-aim}
M.~Galli and M.~Ritor{\'e}.
\newblock Area-stationary and stable surfaces of class {$C^1$} in the
  sub-{R}iemannian {H}eisenberg group $\mathbb{H}^1$.
\newblock \href{http://arxiv.org/abs/1410.3619}{arXiv:1410.3619}, 4 Nov 2014.

\bibitem{MR1404326}
N.~Garofalo and D.-M. Nhieu.
\newblock Isoperimetric and {S}obolev inequalities for
  {C}arnot-{C}arath\'eodory spaces and the existence of minimal surfaces.
\newblock {\em Comm. Pure Appl. Math.}, 49(10):1081--1144, 1996.

\bibitem{MR2976521}
F.~Maggi.
\newblock {\em Sets of finite perimeter and geometric variational problems},
  volume 135 of {\em Cambridge Studies in Advanced Mathematics}.
\newblock Cambridge University Press, Cambridge, 2012.
\newblock An introduction to geometric measure theory.

\bibitem{MR0355766}
U.~Massari.
\newblock Esistenza e regolarit\`a delle ipersuperfice di curvatura media
  assegnata in {$R^{n}$}.
\newblock {\em Arch. Rational Mech. Anal.}, 55:357--382, 1974.

\bibitem{MR1867362}
R.~Montgomery.
\newblock {\em A tour of subriemannian geometries, their geodesics and
  applications}, volume~91 of {\em Mathematical Surveys and Monographs}.
\newblock American Mathematical Society, Providence, RI, 2002.

\bibitem{MR2225631}
S.~D. Pauls.
\newblock {$H$}-minimal graphs of low regularity in {$\Bbb H^1$}.
\newblock {\em Comment. Math. Helv.}, 81(2):337--381, 2006.

\bibitem{MR2435652}
M.~Ritor{\'e} and C.~Rosales.
\newblock Area-stationary surfaces in the {H}eisenberg group {$\Bbb H^1$}.
\newblock {\em Adv. Math.}, 219(2):633--671, 2008.

\end{thebibliography}

\end{document}